\documentclass[reqno, 12pt]{amsart}

\usepackage[algoruled,lined,noresetcount,norelsize]{algorithm2e}

\usepackage{graphicx}
\usepackage{amsmath,amssymb,amsthm}
\usepackage{mathrsfs}
\usepackage{mathabx}\changenotsign
\usepackage{dsfont}
\usepackage{xcolor}
\usepackage{tikz}
\usepackage[backref]{hyperref}
\hypersetup{
    colorlinks,
    linkcolor={red!60!black},
    citecolor={green!60!black},
    urlcolor={blue!60!black}
}
\usepackage[T1]{fontenc}
\usepackage{lmodern}
\usepackage[babel]{microtype}
\usepackage[english]{babel}

\usepackage{geometry}
\numberwithin{equation}{section}
\numberwithin{figure}{section}

\newtheorem{theorem}{Theorem}
\newtheorem{conjecture}[theorem]{Conjecture}
\newtheorem{corollary}[theorem]{Corollary}
\newtheorem{lemma}[theorem]{Lemma}
\newtheorem{observation}[theorem]{Observation}
\newtheorem{problem}[theorem]{Problem}

\newtheorem{remark}[theorem]{Remark}
\newtheorem{claim}[theorem]{Claim}
\newtheorem{definition}[theorem]{Definition}
\newtheorem{proposition}[theorem]{Proposition}

\newcommand{\calH}{\mathcal{H}}
\newcommand{\calF}{\mathcal{F}}
\newcommand{\calV}{\mathcal{V}}

\title{Constructions for positional games and applications to domination games}

\author{Ali Deniz Bagdas, Dennis Clemens, Fabian Hamann, Yannick Mogge}

\begin{document}

\maketitle

\begin{abstract}
We present constructions regarding the general behaviour 
of biased positional games, and amongst others show
that the outcome of such a game can differ in an arbitrary way depending on which player starts the game, and that fair biased games can behave highly non-monotonic.
We construct a gadget that helps to transfer such results to Maker-Breaker domination games, and by this we extend a recent result by Gledel, Ir{\v{s}}i{\v{c}}, and Klav{\v{z}}ar, regarding the length of such games. Additionally, we introduce Waiter-Client dominations games, give tight results when they are played on trees or cycles,
and using our transference gadget we show that in general the length of such games can differ arbitrarily from the length of their Maker-Breaker analogue. 
\end{abstract}


\section{Introduction}

\textbf{Biased Maker-Breaker games.}
Maker-Breaker games, a special kind of positional games, have been studied extensively throughout the last 
decades. For a nice overview we recommend the monograph~\cite{hefetz2014positional} 
as well as the survey~\cite{krivelevich2014positional}. 
In most generality,
let a hypergraph $\calH = (X,\calF)$ and two integers $m,b\geq 1$ be given. The \textit{$(m:b)$ Maker-Breaker game} on $(X,\calF)$ is played as follows. Maker and Breaker alternate in moves,
where in a move Maker claims up to $m$ unclaimed elements of the 
\textit{board} $X$,
and Breaker claims up to $b$ unclaimed elements of $X$.
Maker wins if during the course of the game she manages to claim all elements of a \textit{winning set}, i.e.~a hyperedge  from $\calF$, while Breaker wins otherwise.
Surely, this outcome can depend on who makes the first move. Therefore, whenever it makes a difference in the following, we state clearly whom we assume to be the first player.
If $m=b=1$, the game is called \textit{unbiased}; and otherwise it is called \textit{biased}.

A nice property of $(m:b)$ Maker-Breaker games is that these are monotone with respect to each of the biases $m$ and $b$.
That is, roughly speaking, increasing the bias of one
of the players can never be a disadvantage for this player;
see e.g.~\cite{beck2008combinatorial,hefetz2014positional}. 
This observation leads to the natural definition of 
\textit{threshold biases}, which 
have been studied extensively already and
in many cases have proven to be related to properties of random graphs, see e.g.~\cite{bednarska2000biased,ferber2015generating,gebauer2009asymptotic,krivelevich2011critical,nenadov2023probabilistic}. When both biases get increased simultaneously,
we however cannot expect monotonicity in general, and 
interestingly not so much is known when both biases get increased at the same rate.
For instance, a famous result of Beck~\cite{beck2008combinatorial}, supporting a so-called random graph intuition, states that in the $(1:1)$ Maker-Breaker game on the edges of the complete graph $K_n$,  Maker has a strategy to occupy a clique
of size $(2-o(1))\log_2(n)$, but not larger. 
Still, it remains one of the bigger unsolved open problems stated in Beck's monograph (see Problem 31.1 in \cite{beck2008combinatorial}), whether Maker can do the same in the $(2:2)$ Maker-Breaker game, see also~\cite{gebauer2012clique}.
On the other hand, Balogh et al.~\cite{balogh2009diameter} 
considered the 2-diameter game in which Maker's goal is to occupy a spanning subgraph of $K_n$ with diameter 2, and they proved that Breaker wins the $(1:1)$ variant of this game, while Maker has a winning strategy for the $(2:b)$ variant even when $b\leq \frac{1}{9}n^{1/8}(\log n)^{-3/8}$, provided $n$ is large enough. However, when we look only at fair $(b:b)$ games in this case, we can still observe monotonicity, as Maker wins for all $b\geq 2$ if $n$ is large enough.

In contrast to that, for our first result we prove that fair biased Maker-Breaker games can behave very highly non-monotonic. 

\begin{theorem}
\label{thm:construction.MB.not.monotone}
Let $B\subset \mathbb{N}$ be any finite set. Then there exists a hypergraph $\calH = (X,\calF)$, such that
Maker wins the $(b:b)$ Maker-Breaker game on $\calH$ if and only if $b\notin B$, when Maker is the first player.
\end{theorem}

Next to studying threshold biases, another central topic in positional games is to find fast winning strategies, see e.g.~\cite{clemens2012fast,clemens2020fast,ferber2014weak,ferber2012fast,hefetz2009fast,mikalavcki2018fast}. For many results in these papers, e.g. when Maker aims for a perfect matching, a Hamilton cycle or a fixed spanning tree in the $(1:1)$ game on $K_n$,
the outcome turns out to be independent of who starts the game,
for large enough $n$. Therefore, it becomes a natural question whether the time that Maker needs for winning can differ a lot when we compare the case when Maker starts with the case when Breaker starts. Our next result roughly states that almost anything can happen. Additionally, we are able to give an analogous statement when we consider the smallest winning set that Maker can occupy.

\begin{theorem}\label{thm:construction.MB.time.size}
Let $m,b,s,s',t,t'$ be positive integers
with $s'\geq s\geq 2m+1$, $m\leq b$, $t' \geq t\geq \lceil \frac{s}{m} \rceil$ and $t'\geq \lceil \frac{s'}{m'} \rceil$. Then there exists a hypergraph $\mathcal{H}=(X,\mathcal{F})$ which satisfies the following properties:
\begin{enumerate}
    \item[(1)] Let Maker be the first player in the $(m:b)$ Maker-Breaker game on $\mathcal{H}$. Then she has a strategy to occupy a winning set of size $s$ within $t$ rounds, but she can neither win within less than $t$ rounds nor occupy a winning set of size smaller than $s$.
    \item[(2)] Let Maker be the second player in the $(m:b)$ Maker-Breaker game on $\mathcal{H}$. Then she has a strategy to occupy a winning set of size $s'$ within $t'$ rounds, but she can neither win within less than $t'$ rounds nor occupy a winning set of size smaller than $s'$.
\end{enumerate}
\end{theorem}

Note that all conditions on the parameters
$m,b,s,s',t,t'$ in the theorem above are optimal
except for maybe the inequalities $s\geq 2m+1$ and $m\leq b$. We discuss this in more detail in the concluding remarks.

\medskip

\textbf{Maker-Breaker domination games.}
Many publications on positional games consider the case when
$X$ is the edge set of some given graph $G$, and $\mathcal{F}$
contains the edge sets of e.g. all spanning subgraphs of $G$ that  
can be described by a monotone graph property, such
as being connected~\cite{gebauer2009asymptotic,lehman1964solution}, 
containing a Hamilton cycle~\cite{hefetz2009sharp,krivelevich2011critical}, or similar.
More recently, Duch\^ene et al.~\cite{duchene2020maker} introduced
\textit{Maker-Breaker domination games}, which in contrast to the above games are played on the vertex set of a given graph $G$ by two players
called Dominator and Staller, where Dominator (who is playing as Maker) wins if and only if she manages to
occupy all vertices of a dominating set. 
Note that the renaming of the players Maker and Breaker 
as Dominator and Staller is done to be consistent with the usual domination games; 
for an overview on these games we recommend the book~\cite{brevsar2021domination}.

While Duch\^ene et al.~\cite{duchene2020maker} proved that deciding who wins an unbiased Maker-Breaker domination game
is PSPACE-complete,
Gledel et al.~\cite{gledel2019maker} put a focus on the number of rounds
which Dominator needs to win. Following their
notation, let $\gamma_{MB}(G,m:b)$ denote the smallest number of rounds in which Dominator can always win the $(m:b)$ domination game on $G$, provided that she starts the game, 
and where we set $\gamma_{MB}(G,m:b)=\infty$ if Dominator does not have a winning strategy. Similarly, let
$\gamma_{MB}'(G,m:b)$ denote the smallest number of rounds 
for the case when Staller starts the game,
and for short, let
$\gamma_{MB}(G) := \gamma_{MB}(G,1:1)$ and
$\gamma'_{MB}(G) := \gamma'_{MB}(G,1:1)$.
Gledel et al.~\cite{gledel2019maker} determined
$\gamma_{MB}(G)$ and $\gamma'_{MB}(G)$
precisely when $G$ is a tree or a cycle,
and partial results for 
Cartesian products~\cite{dokyeesun2023maker,forcan2020maker}
and Corona products~\cite{divakaran2024maker}
of graphs
were obtained afterwards as well.
Recently, biased variants of these games were considered in \cite{bagdas2024biased,brevsar2025thresholds}.
Additionally, Gledel et al.~\cite{gledel2019maker}
provided examples which show that the domination number $\gamma(G)$ of a graph $G$
and the Maker-Breaker domination numbers
$\gamma_{MB}(G)$ and $\gamma'_{MB}(G)$ can take almost arbitrary values.

\begin{theorem}\label{thm:gledel.construction}[Theorem~3.1~in~\cite{gledel2019maker}]
If $G$ is a graph, then $\gamma(G)\leq \gamma_{MB}(G)\leq \gamma'_{MB}(G)$. Moreover, for any integers $r,t,t'$ with
$2\leq r\leq t\leq t'$, there exists a graph $G$ such that $\gamma(G)=r$, $\gamma_{MB}(G)=t$ and
$\gamma_{MB}'(G)=t'$.
\end{theorem}

We note that similar constructive results were obtained very recently in~\cite{divakaran2025maker} for \textit{total domination games}, in which Maker's goal is to occupy a total dominating set.

\smallskip

Next to considering the number of rounds, it seems natural to ask for the smallest size of a dominating set that Dominator can
always claim when playing the $(m:b)$ domination game on a graph $G$. Let $s_{MB}(G,m:b)$ and $s_{MB}(G)=s_{MB}(G,1:1)$ denote this values when Dominator starts the game, and let $s'_{MB}(G,m:b)$ and $s'_{MB}(G)=s_{MB}(G,1:1)$ denote this value when Staller starts, where again we set such a value to $\infty$ if Dominator does not have a winning strategy.
A priori it is not clear why 
the minimal number of rounds and the minimal size 
of a dominating set that Dominator can achieve
should be different. 
In fact, from the proofs in~\cite{gledel2019maker} 
it can be deduced easily that
$\gamma_{MB}(G)=s_{MB}(G)$ and
$\gamma'_{MB}(G)=s'_{MB}(G)$ hold when 
$G$ is a tree or a cycle.
Moreover, it is not hard to check that
if $G$ is the Erd\H{o}s-R\'enyi random graph $G(n,p)$ with constant $p\in (0,1)$ or a random geometric graph $\mathcal{G}(n,r)$ with constant radius $r\in (0,1)$, then with high probability the parameters
$\gamma_{MB}(G)$ and $s_{MB}(G)$ are asymptotically the same (and maybe equal), see also the concluding remarks.
However, in contrast to this, we can construct graphs where the mentioned game parameters can take almost arbitrary values, hence extending Theorem~\ref{thm:gledel.construction}.

\begin{theorem}
\label{thm:construction.Gledel.improved}
Let $m,b,r,s,s',t,t'$ be positive integers
with $s'\geq s\geq \max\{r,2m+1\}$, $m\leq \min\{b,r-1\}$, $t' \geq t\geq \lceil \frac{s}{m} \rceil$ and $t'\geq \lceil \frac{s'}{m'} \rceil$. Then
there exists a graph $G$ such that 
\begin{align*}
& \gamma(G)=r, \\
& s_{MB}(G,m:b)=s ~ \text{ and } ~ \gamma_{MB}(G,m:b)=t\\
& s'_{MB}(G,m:b)=s' ~ \text{ and } ~ \gamma'_{MB}(G,m:b)=t' .
\end{align*}
\end{theorem}

We again discuss the optimality of our conditions in Section~\ref{sec:concluding}. 
Also note that this theorem implies Theorem~\ref{thm:construction.MB.time.size}, but later we actually need Theorem~\ref{thm:construction.MB.time.size}
for the proof of Theorem~\ref{thm:construction.Gledel.improved}.
In fact, we construct a gadget (see Lemma~\ref{lemma:from.H.to.G}) which allows us to transfer 
constructive results from general positional games
to domination games. 

With the help of the same transference argument it is possible to prove an analogue of Theorem~\ref{thm:construction.MB.not.monotone} for Maker-Breaker domination games. Additionally, we are able to prove the following rather non-intuitive results which demonstrate that Dominator might need to employ very different strategies depending on whether she wants to win as fast as possible or claim a dominating set of smallest possible size.
The first result roughly states that 
Dominator can be forced to claim an arbitrarily large 
dominating set before she is able to claim a smallest possible
dominating set, i.e. she might need to play on after she already claimed a dominating set if she wants to minimize the size.

\begin{theorem}
\label{thm:construction.large.before.small}
For any biases $m \leq b$ and any integers $t> s \geq 2m + 1$,
there exists a graph $G$ such that
$s_{MB}(G,m:b)=s$, but Dominator cannot occupy
a dominating set of size $s$ before she has  
occupied another minimal dominating set of size $t$.
\end{theorem}

The second result states that there are graphs for which Dominator can only achieve optimal time if she claims an arbitrarily large dominating set, while the smallest possible dominating set can take arbitrarily many rounds for her to claim, i.e. there are cases where Dominator needs to decide which goal she wants to pursue, since she cannot achieve both with the same strategy.

\begin{theorem}
\label{thm:construction.large.faster.than.small}
For any biases $m \leq b$ and any integers 
$s'\geq s\geq 2m + 1$, and $t' \geq t \geq \lceil\frac{s'}{m}\rceil$
there exists a graph $G$ such that
$\gamma_{MB}(G,m:b)=t$ and $s_{MB}(G,m:b)=s$, but
in the $(m:b)$ Maker-Breaker domination game on $G$ Breaker has a strategy such that all of the following hold:
\begin{itemize}
 \item Maker cannot claim a dominating set of size smaller than $s$.
 \item Maker cannot win in less than $t$ rounds.
 \item Maker cannot claim a dominating set of size smaller than $s'$ if she wins in less than $t'$ moves.
\end{itemize}
\end{theorem}

\medskip

\textbf{Waiter-Client domination games.}
In the last years, \textit{Waiter-Client games} (earlier called Picker-Chooser games, see e.g.~\cite{beck2008combinatorial}) have received increasing attention, ranging from results on fast winning strategies~\cite{clemens2020fast,dvovrak2021waiter} 
over biased games~\cite{bednarska2016manipulative,hefetzCW,nenadov2023probabilistic} to games played on random graphs~\cite{clemens2021waiter,hefetz2017waiter}. 
In the following we stick to the case of unbiased Waiter-Client games only. On a given hypergraph $\calH = (X,\calF)$, these games are played almost
the same way as Maker-Breaker games with the following difference: In every round, Waiter chooses two unclaimed elements of the board $X$ and then Client decides which 
of these elements goes to Waiter while the other one goes to Client. (If in the last round there is only
one unclaimed element left, then it is claimed by Client.) 
Waiter wins if and only if she manages to claim all 
elements of a winning set from $\calF$.\footnote{In some papers,
Waiter-Client games are defined so that Waiter's goal is to force Client to claim a winning set; and then accordingly if there is a unique element in the last round, it goes to Waiter. Both versions are essentially the same.}

So far, domination games have not been studied in this setting. So, we aim to give first results for Waiter-Client domination games and also study how Waiter-Client and Maker-Breaker domination games are related.
Given a graph $G$, we define the \textit{Waiter-Client domination game} on $G$ in the obvious way: Dominator (playing as Waiter) offers two unclaimed vertices of $G$ and then 
Staller (playing as Client) picks one of these vertices for himself and the other goes to Dominator.
In accordance with previous notation,
we denote with 
$\gamma_{WC}(G)$
the smallest number of rounds in which
Dominator can always occupy a dominating set
in the Waiter-Client domination game on $G$,
and we let $s_{WC}(G)$ denote the size of the smallest
dominating set that Waiter can always claim.
For our first results in this setup,
we see that for cycles and trees
the game behaves the same way
as in the Maker-Breaker setting~\cite{gledel2019maker} (when Breaker starts).

\begin{theorem}\label{thm:WC.trees}
Let $T$ be a tree on $n$ vertices.
If $T$ has a perfect matching, then
$$\gamma_{WC}(T) = s_{WC}(T) = \frac{n}{2} .$$
In all other cases, Dominator does not win the $(1:1)$
Waiter-Client domination game on $T$.
\end{theorem}

\begin{theorem}\label{thm:WC.cycle}
For every $n\geq 3$,
$$
\gamma_{WC}(C_n) = s_{WC}(C_n) = \left\lfloor \frac{n}{2} \right\rfloor .
$$
\end{theorem}

Due to these results and due to the fact that Waiter-Client games often tend to be easier wins for Waiter than
for Maker in the corresponding Maker-Breaker games,
one could ask whether a relation such as
$\gamma_{WC}(G) \leq \gamma'_{MB}(G)$
can be proven for any graph $G$.
By using our transference lemma
together with suitable constructions of 
hypergraphs $\calH = (X,\calF)$ for Maker-Breaker and Waiter-Client games, we answer this question negatively as follows. 

\begin{theorem}\label{thm:construction.MB.and.WC}
For all integers $s,t\geq 7$ there is a graph $G$ such
that $\gamma'_{MB}(G)=s$ and $\gamma_{WC}(G)=t$.
\end{theorem}

\smallskip

\textbf{Organization of the paper.}
In Section~\ref{sec:transference} we prove our central lemma which helps to transfer general results on positional games
into results on domination games.
Afterwards, in Section~\ref{sec:constructions},
we prove Theorem~\ref{thm:construction.MB.not.monotone}
and Theorem~\ref{thm:construction.MB.time.size},
and building on the transference lemma, we prove Theorem~\ref{thm:construction.Gledel.improved}, Theorem~\ref{thm:construction.large.before.small} and Theorem~\ref{thm:construction.large.faster.than.small}. 
In Section~\ref{sec:WC.domination} we then prove
all our results for the Waiter-Client setting:
Theorem~\ref{thm:WC.trees}, Theorem~\ref{thm:WC.cycle}
and Theorem~\ref{thm:construction.MB.and.WC}.

\medskip

\textbf{Notation.} We use standard notation.
Given a graph $G$, we write $V(G)$ for its vertex set and $v(G)=|V(G)|$, and we write $G[X]$ for the subgraph induced on a vertex set $X$. Given a vertex $v$, we let $N_G(v)$ be the neighourhood of $v$ in $G$ and set $d_G(v)=|N_G(v)|$ for the degree of $v$.
Moreover, given any subset $A\subset V(G)$, we set
$G-A=G[V(G)\setminus A]$ for short.


\section{From hypergraphs to domination games}\label{sec:transference}

In order to describe our transference lemma,
we make use of the following definitions.

\begin{definition}\label{def:short.notation}
Let $\mathcal{H}$ be a hypergraph, 
let $G$ be a graph,
let $a,b,s,t\in\mathbb{N}$. Then we write
\begin{itemize}
\item $\mathcal{H} \in M_i(a:b,s,t)$ if
Maker has a strategy to claim a winning set
of size at most $s$ within at most $t$ rounds
in the $(a:b)$ Maker-Breaker game on $\mathcal{H}$,
when Maker is the $i^{\text{th}}$ player.
\item $\mathcal{H} \in W(s,t)$ if
Waiter has a strategy to claim a winning set
of size at most $s$ within at most $t$ rounds
in the unbiased Waiter-Client game on $\mathcal{H}$.
\item $G \in D_{MB,i}(a:b,s,t)$ if
Dominator has a strategy to occupy a dominating set
of size at most $s$ within at most $t$ rounds
in the $(a:b)$ Maker-Breaker domination game on $G$,
when Dominator is the $i^{\text{th}}$ player.
\item $G \in D_{WC}(s,t)$ if
Dominator has a strategy to occupy a dominating set
of size at most $s$ within at most $t$ rounds
in the unbiased Waiter-Client domination game on $G$.
\end{itemize}

\end{definition}

\begin{definition}\label{def:transference.construction}
Let $\calH = (X,\calF)$ be a hypergraph,
let $a\geq 1$ be an integer.
Let $\calV(\calH) \subset \mathcal{P}(X)$ be the set of all vertex covers of $\calH$, and for each $A\in \calV(\calH)$, 
let $X_A$ be a vertex set of size $4a(|X| + |\calV(\calH)|)$,
such that all the sets $X_A$ and $X$ are pairwise disjoint.
We construct a graph $G=G(\calH,a)$
as follows:
\begin{itemize}
\item its vertex set is 
$V(G)=X\cup (\bigcup_{A\in \calV(\calH)} X_A)$, and
\item its edge set is $E(G)=\{vw:~ (\exists A\in \calV(\calH):~v\in A ~ \land ~ w\in X_A)\} \cup \binom{X}{2}$.
\end{itemize}
\end{definition}

\begin{lemma}[Transference Lemma]\label{lemma:from.H.to.G}
Let $\calH =(X,\mathcal{F})$ be a hypergraph, 
let $a,b,s,t\in\mathbb{N}$, and let $G=G(\mathcal{H},a)$ be constructed according to Definition~\ref{def:transference.construction}. 
Then the following holds:
\begin{enumerate}
\item[(M)] $\mathcal{H} \in M_i(a:b,s,t) ~ \Leftrightarrow ~ G \in D_{MB,i}(a:b,s,t)$,
\item[(W)] $\mathcal{H} \in W(s,t) ~ \Leftrightarrow ~ G \in D_{WC}(s,t)$.
\end{enumerate}
\end{lemma}

Note that in (W) we stick only to the unbiased game as this is all we need for our later results; but a biased version would also be possible.

Before we prove Lemma~\ref{lemma:from.H.to.G} we first collect
useful claims. From now on, let $\calH$ and $G=G(\calH,a)$
be fixed according to the assumptions of
Lemma~\ref{lemma:from.H.to.G}.

\begin{claim}\label{clm:transference.claim1}
Each $F\in \calF$ is a dominating set of $G$.
\end{claim}

\begin{proof}
Let $v\in V(G)\setminus F$.
If $v\in X$, then it is dominated by $F$ since $G[X]$ is a complete graph.
If $v\notin X$, then there is a set $A\in \calV(\calH)$
such that $v\in X_A$, and therefore $A=N_G(v)$. Since $A$ is a vertex-cover of $\mathcal{H}$, it holds that
$A\cap F\neq \varnothing$, i.e. $F$ contains a neighbour of $v$.
\end{proof}

\begin{claim}\label{clm:transference.claim2}
Let $k:=\min\{|F|:~F\in\mathcal{F}\}$.
The domination number of $G$ is $\gamma(G)=k$.
\end{claim}

\begin{proof}
By Claim~\ref{clm:transference.claim1} we have $\gamma(G)\leq k$. Now, for contradiction, assume that there is dominating set $D$ of size smaller than $k$. Then $A:=X\setminus D$ is a vertex-cover of $\calH$, since every hyperedge $F\in \calF$ satisfies $|F|>|D|$. 
But then no vertex in $X_A$ is dominated by $X\cap D$.
So, if $D$ is a dominating set, we then need
$X_A\subset D\setminus X$. But this implies
$|D|\geq |X_A| > 4|X| \geq 4k$, a contradiction.
\end{proof}

\begin{claim}\label{clm:transference.claim3}
If $D\subset X$ is a dominating set in $G$, then there is some $F\in\mathcal{F}$ such that $F\subset D$.
\end{claim}

\begin{proof}
Let $D\subset X$ be a dominating set in $G$, and assume there is no $F\in\mathcal{F}$ such that $F\subset D$. Then each $F\in \calF$ intersects $A:=X\setminus D$
and hence $A$ is a vertex-cover of $\calH$. But then, no vertex in $X_A$ has a neighbour in $D$, a contradiction.
\end{proof}

With the above claims in hand, we can prove
Lemma~\ref{lemma:from.H.to.G}.

\smallskip

\begin{proof}[Proof of Lemma~\ref{lemma:from.H.to.G}]

\textbf{(M)} 
"$\Rightarrow$" 
Assume that $\mathcal{H} \in M_i(a:b,s,t)$, i.e.~Maker (being the $i^{\text{th}}$ player) has a strategy $\mathcal{S}_M$ to claim a winning set of size at most $s$ within at most $t$ rounds in the $(a:b)$ Maker-Breaker game on $\mathcal{H}$.
Then in the $(a:b)$ Maker-Breaker domination game on $G$, Dominator (being the $i^{\text{th}}$ player) can restrict her moves to $G[X]$ and play according to the strategy $\mathcal{S}_M$. 
This way, Dominator claims a set $F\in \mathcal{F}$ of size at most $s$ within at most $t$ rounds, and thus by Claim~\ref{clm:transference.claim1}, we have $G \in D_{MB,i}(a:b,s,t)$.

"$\Leftarrow$" We do a proof by contraposition. Assume that 
$\mathcal{H} \notin M_i(a:b,s,t)$, and therefore
Breaker has a strategy $\mathcal{S}_B$ 
in the $(a:b)$ Maker-Breaker game on $\calH$ 
(where Maker is the $i^{\text{th}}$ player)
which prevents Maker from claiming a winning set of size
at most $s$ within the first $t$ rounds. Then, in the $(a:b)$ Maker-Breaker domination game on $G$
(where Dominator is the $i^{\text{th}}$ player),
Staller can apply the following strategy:
\begin{itemize}
\item[] \textbf{Stage I:} As long as before Staller's move there is at least one free vertex in $X$,
Staller plays on $G[X]$ only, and he plays according to Breaker's 
strategy $\mathcal{S}_B$. (If Dominator claims an element
outside of $X$, Staller simply pretends that Dominator made an arbitrary move within $X$.)
\item[] \textbf{Stage II:} When Staller enters Stage II, all elements of $X$ are claimed. For each $A\in \calV(\calH)$,
let $X_A'\subset X_A$ be the vertices which are still free at the end of Stage I. Then from now on, for each vertex that Staller can claim,
Staller chooses a vertex from an arbitrary set $X_A'$ in which he still has no vertex claimed.
\end{itemize}

Note that Stage I lasts at most $|X|$
rounds, and in the meantime Dominator cannot claim more than $a|X|$ vertices. Hence, for 
Stage II of Staller's strategy we have $|X_A'|\geq 4a|\calV(\calH)|$ for every $A\in \calV(\calH)$.
Stage II lasts at most $|\calV(\calH)|$ rounds,
since Staller only wants to claim one vertex in each of the sets $X_A'$. Hence, during this whole stage Dominator does not claim more than $a|\calV(\calH)|$ vertices, which implies that at any point in Stage II, there are at least $2a|\calV(\calH)|$
free vertices in each set $X_A'$ and thus
Staller can always claim a free vertex as required by Stage II of the strategy.

It remains to prove that, by following the strategy, Staller 
prevents Dominator from getting a dominating set of $G$ 
of size at most $s$ within
the first $t$ rounds.
For contradiction, assume the opposite, i.e. 
Dominator claims a dominating set $D$ with $|D|\leq s$ within at most $t$ rounds.
Then $D' := D\cap X$ would be a dominating set by the following reason:
Because of Stage II, for each $A\in \mathcal{V}(\calH)$,
there is vertex $v\in X_A$ such that $v\notin D$.
Hence, since $N_G(v)=A\subset X$, 
we get $D\cap A \neq \varnothing$
and therefore $D'\cap A \neq \varnothing$. 
Thus, $D'$ dominates every vertex in $X_A$ for every $A\in \mathcal{V}(\calH)$,
and it also dominates every vertex in $X\setminus D' = X\setminus D$
since $G[X]$ is a complete graph. \\
By Claim~\ref{clm:transference.claim3}
it follows that there is some $F\in \calF$ with 
$F\subset D'\subset D$,
and note that $|F|\leq |D|\leq s$ by assumption. 
However, by following Stage~I, Staller ensures that Dominator needs to play at least $t$ rounds
on $X$ to occupy a winning set $F\in \calF$
of size at most $s$, a contradiction.

\medskip

\textbf{(W)} "$\Rightarrow$" The proof is essentially the same as for (M).

"$\Leftarrow$" The proof is very similar to the corresponding proof for (M).
Assume that Client has a strategy $\mathcal{S}_C$ in the 
Waiter-Client game on $\calH$ which prevents Waiter from 
getting a winning set of size at most $s$ 
within the first $t$ rounds. 
Now, in the Waiter-Client domination game on $G$,
Staller can apply the following strategy which distinguishes
three cases.
\begin{itemize}
\item[] \textbf{Case I:} When Dominator offers two elements of $X$,
then Staller distributes these vertices according to the strategy $\mathcal{S}_C$.
\item[] \textbf{Case II:} When Dominator offers a vertex $v\in X$ and a vertex $w \notin X$, then Staller claims $v$ and gives $w$ to 
Dominator. (For the strategy $\mathcal{S}_C$ in Case I, Client imagines that $v$ was not played yet.)
\item[] \textbf{Case III:} When Dominator offers two vertices 
$v\in X_{A_1}$ and $w\in X_{A_2}$ with $A_1,A_2\in \calV(\calH)$,
then Staller claims $v$ if he does not have a vertex in $X_{A_1}$ yet, and otherwise Staller claims $w$. The remaining vertex goes to Dominator.
\end{itemize}

By the strategy $\mathcal{S}_C$ in Case I,
and since Dominator never gets an element of $X$ in the other cases,
Dominator needs more than $t$ rounds in order to claim
a winning set of $\mathcal{F}$ with size at most $s$.
Moreover, by the strategy from Cases~II and~III,
Staller makes sure to receive at least one vertex in each set $X_A$.
For this note that for each $A\in \calV(\calH)$ there can be at 
most $|X|$ rounds of Case II and at
most $|\calV(\calH)|$ rounds of Case III, in which a vertex $x\in X_A$ is offered while Staller still does not have a vertex in $X_A$, but Staller decides to claim the other vertex which was offered by Dominator.

Now, using these observation, the same discussion as at the end of the proof of (M) gives that
Staller prevents Dominator from occupying a dominating set of $G$ of size at most $s$ within the first $t$ rounds.
\end{proof}


\section{Constructions}\label{sec:constructions}

\subsection{A Maker-Breaker like game on directed multigraphs}

As a first building block towards our constructions for Theorems~\ref{thm:construction.MB.time.size} and ~\ref{thm:construction.Gledel.improved}--\ref{thm:construction.large.faster.than.small}, we define certain directed multigraphs $\vec{G}_t(b)$ and
$\vec{H}_t(b)$, and also analyze a suitably defined positional game on these directed multigraphs. We start with the definition of $\vec{G}_t(b)$.

\begin{definition}\label{def:directed.G}
Let positive integers $b,t$ be given. We iteratively define a directed multigraph $\vec{G}_t(b)$ with startpoint $v_0$ and endpoint $v_t$ as follows:
\begin{enumerate}
\item Let $\vec{G}_1(b)$ consist of two vertices $v_0$ and $v_1$,
and one directed edge going from $v_0$ to $v_1$.
\item Given $\vec{G}_i(b)$ for some $1\leq i\leq t-1$, with endpoints $v_0$ and $v_i$, we obtain $\vec{G}_{i+1}(b)$ as follows:
\begin{enumerate}
\item add a new vertex $v_{i+1}$,
\item add $b$ copies of $\vec{G}_i(b)$, each starting in $v_i$ and ending in $v_{i+1}$, such that these copies intersect only in $v_i$ and $v_{i+1}$ and they intersect the initial graph $\vec{G}_i(b)$ only in $v_i$.
\end{enumerate}
\end{enumerate} 
\end{definition}

For an example, see Figure~\ref{fig:constructions}.
Note that by a simple induction on $t$ it can be seen that all vertices of $\vec{G}_t(b)$ except from $v_0$ and $v_t$ have outdegree $b$.

\begin{figure}[ht]
    \centering
\includegraphics[scale=0.7,page=1]{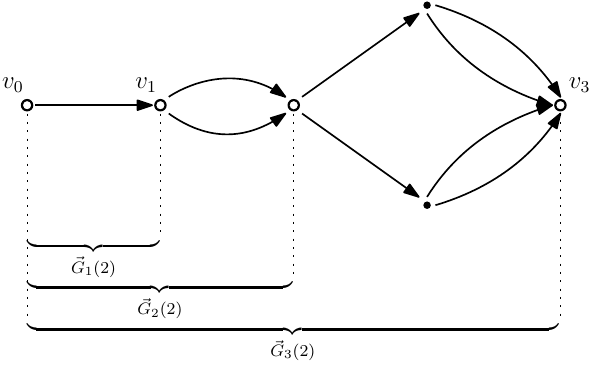}
    \caption{$\vec{G}_t(2)$ for $t\in\{1,2,3\}$}
    \label{fig:constructions}
\end{figure}

\begin{definition}\label{def:directed.H}
Let positive integers $b,t$ with $t\geq 3$ be given. 
We construct $\vec{H}_t(b)$ as follows:
We take $b+2$ vertices $x_0,\dots,x_{b+1}$. 
For every $i \in [b+1]$, we add $b+1$ different copies of 
$\vec{G}_{t-2}(b)$ with startpoint $x_0$ and endpoint $x_i$, such that these copies intersect only in $x_{0}$ and $x_i$ and use new vertices otherwise.
\end{definition}

We next define an auxiliary game on 
directed multigraphs and analyze how fast Maker can win this game
on $\vec{G}_t(b)$ or $\vec{H}_t(b)$. 

\begin{definition}[Auxiliary game $\mathcal{G}(b,\vec{F},S)$]
\label{def:auxiliary.directed.game}
Let a directed multigraph $\vec{F}$ and a subset
$S \subseteq V(\vec{F})$ be given.
The game $\mathcal{G}(b,\vec{F},S)$ is a $(1:b)$ Maker-Breaker game  played on the vertices and the directed edges of $\vec{F}$,
where Maker is the first player, with the following adjustments:
\begin{itemize}
\item[(a)] Maker can only claim a directed edge 
$(v,w)$ of $\vec{F}$ if she already claimed both vertices $v,w$.
\item[(b)] At the beginning of the game, the vertices in $S$ already belong to Maker.
\end{itemize}
Maker wins the game $\mathcal{G}(b,\vec{F},S)$ if she manages to claim a directed edge; Breaker wins otherwise.
\end{definition}

We note that for the description of the game above we could actually replace
directed edges with simple edges. However, keeping these directions allows us later to define a partial ordering on the vertex set for the description and analysis of suitable strategies.
We start with the following lemma.

\begin{lemma} \label{lem:game.directed.G}
Let positive integers $b,t$ be given.
Let $v_0$ and $v_t$ be the start- and endpoint of $\vec{G}_t(b)$. 
Then the following holds:
\begin{enumerate}
\item Maker has a strategy to win
$\mathcal{G}(b,\vec{G}_t(b),\{v_0,v_t\})$ within $t$ rounds.
\item Breaker has a strategy to win
$\mathcal{G}(b,\vec{G}_t(b),\{v\})$ for any vertex $v$.
\item Breaker has a strategy to prevent Maker from winning
$\mathcal{G}(b,\vec{G}_t(b),\{v_0,v_t\})$ within $t-1$ rounds.
\end{enumerate}
\end{lemma}

\begin{proof}
Throughout the proof, if $x,y$ are vertices in $\vec{G}_t(b)$, we write $x \preccurlyeq y$ if there exists a
directed path starting in $x$ and ending in $y$. Note that this defines a partial ordering on $V(\vec{G}_t(b))$.

\medskip

We first prove property (1) by induction on $t$.
For $t=1$, Maker already possesses $v_0$ and $v_1$ before the game starts, and hence she can win in the first round by claiming the directed edge $(v_0,v_1)$.
For $t>1$, Maker possesses $v_0$ and $v_t$ before the game starts. For her first move, she claims the vertex $v_{t-1}$.
As there are $b+1$ copies of $\vec{G}_{t-1}(b)$,
that have their start- and endpoints in $\{v_0,v_{t-1},v_t\}$
and which are pairwise disjoint outside of this set,
and since Breaker's bias is $b$,
after Breaker's first move there must be at least
one copy of $\vec{G}_{t-1}(b)$ with no Breaker vertices,
but with the start- and endpoint claimed by Maker. By induction
Maker has a strategy to win within the next $t-1$ rounds,
hence proving the statement (1).

\medskip

Next, consider statement (2).
We prove by induction on the number of rounds that in the game 
$\mathcal{G}(b,\vec{G}_t(b),\{v\})$, Breaker can play in such a way that after each of his moves the following properties hold:
\begin{enumerate}
\item[(P1)] There is at most one Maker's vertex $x$ with free outgoing edges. 
\item[(P2)] For all vertices $y\neq z$ with $y \preccurlyeq z$: If $y$ and $z$ are claimed by Maker, then Breaker has claimed all outgoing edges at $y$.
\end{enumerate}
Note that (P1) and (P2) are true before the game starts, as Maker possesses only the vertex $v$ at that point. For induction, assume that (P1) and (P2) hold after the $(i-1)$-st round for some $i\in \mathbb{N}$, and let $w$ be the vertex that Maker claims in round $i$. If all of Maker's vertices except from
$w$ have no free outgoing edges, Breaker simply claims all outgoing edges at $w$, and (P1) and (P2) follow immediately.
Hence, using (P1) from the induction hypothesis, we may assume from now on that there is a unique vertex $w'\neq w$
that was claimed by Maker before the $i$-th round, and which still has free outgoing edges.
If $w' \preccurlyeq w$,
then Breaker claims all outgoing edges at $w'$. Then (P1) is maintained, as $w$ is the only Maker's vertex that can have free outgoing edges. Also, due to this, if (P2) would fail to hold
at the end of the $i$-th round, then we would need $w \preccurlyeq z$ for some Maker's vertex $z$ that was claimed in an earlier round. But then also $w' \preccurlyeq z$ would need to hold by transitivity, while $w'$ had free outgoing edges at the end of the $(i-1)$-st round, in contradiction to (P2) at the end of the $(i-1)$-st round. Hence, (P2) is maintained.
If otherwise $w' \preccurlyeq w$ does not hold, then Breaker claims all outgoing edges at $w$. Then (P1) remains true, with $w'$ staying the unique Maker's vertex with free outgoing edges.
If (P2) would fail to hold at the end of the $i$-th round, 
then we would need $w' \preccurlyeq z$ for some Maker's vertex $z$. As (P2) was true at the end of the $(i-1)$-st round,
$z$ cannot be claimed until the $(i-1)$-st round, implying
$z=w$, which however is in contradiction to the assumption of this case. Thus, (P2) holds again, finishing the argument that Breaker can always maintain both properties.

Finally, note that (P2) implies that Maker can never claim a directed edge $(y,z)$, since when $y$ and $z$ are claimed by Maker, Breaker claims all outgoing edges at $y$.

\smallskip

It remains to prove (3) about the game
$\mathcal{G}(b,\vec{G}_t(b),\{v_0,v_t\})$. We again use induction on the number of rounds and prove that Breaker has a strategy that after his $i$-th move the following properties hold for all $i \in [t - 1]$:
\begin{enumerate}
\item[(P1')] There is at most one Maker's vertex $x$ with free outgoing edges. 
\item[(P2')] For all vertices $y\neq z$ with $y \preccurlyeq z$: If $y$ and $z$ are claimed by Maker and there is a directed path from $y$ to $z$ of length less than $2^{t - i - 1}$, then Breaker has claimed all outgoing edges at $y$.
\end{enumerate}

At the start of the game (P1') and (P2') hold by definition. The only vertex with free outgoing edges is $v_0$, and the shortest directed path between $v_0$ and $v_t$ is by construction of length $2^{t-1}$.

Now assume (P1') and (P2') hold after Breaker's $(i-1)$-th move for some $i \in \mathbb{N}$. Let $x$ be the single vertex with free outgoing edges claimed by Maker. For every other vertex $z$ with $x \preccurlyeq z$ claimed by Maker we know that the shortest directed path from $x$ to $z$ is of length at least $2 ^ {t - (i-1) - 1}$, because otherwise Breaker would have claimed all outgoing edges of $x$ by property (P2').

Let $y$ be the vertex claimed by Maker in her $i$-th move. If $x \not\preccurlyeq y$, Breaker can claim all outgoing edges at $y$, and (P1') and (P2') hold as before. Otherwise, let $\ell$ be the length of a shortest directed path from $x$ to $y$. If $\ell \geq 2 ^ {t - i - 1}$, Breaker again claims all outgoing edges at $y$, such that (P1') and (P2') hold true afterwards. Otherwise if $\ell < 2 ^{t - i - 1}$, he claims all outgoing edges at $x$ such that (P1') still trivially holds. For (P2'), it remains to show that there is no short directed path from $y$ to some vertex $z$ with $y \preccurlyeq z$ claimed by Maker. Assume there is such a path of length less than $2 ^ {t - i - 1}$ to a vertex $z$. Then there was a directed path of length less than $2 ^{t - (i - 1) - 1}$ from $x$ to $z$ after Breaker's $(i-1)$-th move contradicting the assumption of (P2'). 

To conclude we observe that Maker can only win in her turn if there are two vertices $y$ and $z$ claimed by her such that there is a directed path of length $1$ from $y$ to $z$ which is not blocked by Breaker. Thus, because (P2') holds throughout the game, Maker cannot win in less than $t$ moves.
\end{proof}

\begin{lemma} \label{lem:game.directed.H}
Let positive integers $b,t$ with $t\geq 3$ be given.
Let $\vec{H}_t(b)$ and the vertices
$x_0,\ldots,x_{b+1}$ be given according to Definition~\ref{def:directed.H}. 
Then the following holds for the game 
$\mathcal{G}(b,\vec{H}_t(b),\varnothing)$:
\begin{enumerate}
		\item Maker has a strategy to win within $t$ rounds.
		\item Breaker has a winning strategy if he can claim a single vertex before Maker starts with her first move.
		\item Breaker has a strategy to prevent Maker from winning within $t-1$ rounds.
	\end{enumerate}
\end{lemma}

\begin{proof}
We first prove (1).	Maker claims $x_0$ in her first move. 
Since Breaker can only claim $b$ elements in his move, there has to be an $i \in [b+1]$ such that Breaker claimed no vertices from
the copies of $\vec{G}_{t-2}(b)$ that start in $x_0$ and end in $x_i$. Then Maker claims $x_i$, and analogously we have that after Breaker's next move there must be a copy of $\vec{G}_{t-2}(b)$
from $x_0$ to $x_i$ that Breaker did not touch yet. By
Lemma~\ref{lem:game.directed.G}(1), playing only on this copy, Maker then has a strategy
to win the game within the next $t-2$ rounds, proving (1).

\smallskip

To prove (2), Breaker claims $x_0$ as the single element before Maker's first move. 
Afterwards, we may pretend that Maker already claimed all vertices $x_1 \dots x_{b+1}$; then
every free vertex belongs to a unique copy of 
$\vec{G}_{t-2}(b)$. On these copies, Breaker can then
separately apply the strategy guaranteed by Lemma~\ref{lem:game.directed.G}(2),
and hence win the game.
	
	\smallskip

It remains to prove (3). We distinguish two cases depending on Maker's first move. First, we consider the case when Maker does not take $x_0$ in her first move, and show that Breaker actually wins the game in that case.
\begin{claim}
    If Maker does not claim $x_0$ in her first move, Breaker has a strategy to win the game.
\end{claim}
\begin{proof}
    Independent of Maker's first move (except $x_0$), Breaker will claim $x_0$ in his first move (and arbitrary other vertices or edges if $b > 1$). Afterwards, Breaker claims in each of his moves all outgoing edges of a vertex claimed by Maker which is possible because the outdegrees are bounded by $b$. By doing so, he makes sure that after his move there will be at most one vertex claimed by Maker with free outgoing edges. If Maker ever claims a vertex without outgoing edges, i.e. a vertex from $\{x_1,\ldots,x_{b+1}\}$, there will be no Maker's vertices with free outgoing edges after Breaker's next move, and after all further Breaker moves as well. This is even the case if Maker claims a vertex from $\{x_1,\ldots,x_{b+1}\}$ in her first move. Since Maker can only win if there is a Maker's vertex with free outgoing edges at the start of her turn, she will not be able to win anymore if she ever claims a vertex from $\{x_1,\ldots,x_{b+1}\}$.
    Thus, we assume she never claims a vertex from $\{x_0,\ldots,x_{b+1}\}$. Each vertex $v \notin \{x_0,\ldots,x_{b+1}\}$ belongs to a unique copy $\vec{G}$ of $\vec{G}_{t-2}(b)$. When Maker claims $v$ in any turn besides the first, Breaker answers on $\vec{G}$ according to his strategy from Lemma~\ref{lem:game.directed.G}(2) which allows him to win on $\vec{G}$ even if Maker already claimed a vertex of $\vec{G}$ in her first move. Moreover, in the strategy of Lemma~\ref{lem:game.directed.G}(2), Breaker claims in each of his turns all outgoing edges of a Maker's vertex as required.
\end{proof}
Hence, from now on we can assume that Maker's first move is to claim $x_0$. In this case, Breaker could actually skip his first move, so he claims arbitrary edges or vertices and pretends that he did not claim anything yet. Whenever his strategy leads him to claim an element he already possesses, he claims an arbitrary vertex instead.
We distinguish two cases depending on Maker's second move.
If Maker claims $x_i$ for some $i \in [b+1]$ in her second move, then afterwards Breaker can pretend that Maker actually claimed $x_1,\ldots,x_{b+1}$ and applies the strategy from Lemma~\ref{lem:game.directed.G}(3) to every copy of $\vec{G}_{t-2}(b)$, to ensure that Maker needs at least $t-2$ further rounds for winning, as required. In this case, Breaker's second move is again arbitrary and could be skipped.
Otherwise, assume that Maker claims a vertex $v\notin \{x_0,\ldots,x_{b+1}\}$ in her second move, i.e. a vertex $v$ that again belongs to a unique copy $\vec{G}_v$ of $\vec{G}_{t-2}(b)$ from $x_0$ to a vertex $x_i$, $i\in [b+1]$. Whenever Maker plays on $\vec{G}_v$ (including $v$ in her second move, and potentially $x_i$ in a future move), Breaker answers with his strategy from Lemma~\ref{lem:game.directed.G}(2), preventing her from winning on $\vec{G}_v$. Whenever Maker plays on any other copy $\vec{G}$ of $\vec{G}_{t-2}(b)$ from $x_0$ to some $x_i$, $i \in [b+1]$, Breaker pretends that Maker already claimed the corresponding $x_i$, and responds on $\vec{G}$ according to his strategy from Lemma~\ref{lem:game.directed.G}(3) making sure that Maker cannot win on $\vec{G}$ in less than $t-2$ rounds. Thus, she cannot win in less than $t$ rounds as required, since she already claimed $x_0$ and $v$ in her first two moves.
\end{proof}

\smallskip

\subsection{From digraphs to hypergraphs}\label{sec:branched.paths}
Getting slightly closer to our final constructions, we now start creating suitable hypergraphs based on the directed
multigraphs $\vec{H}_t(b)$. 

\begin{definition}\label{def:hypergraph.H}
	Let $m,b,s,t \in \mathbb{N}$ with $s \geq 2m+1$, $m \leq b$, and $t \geq \lceil \frac{s}{m} \rceil$.
	We define an $s$-uniform hypergraph
	$\mathcal{H}(m,b,s,t)$ 
    and a family $\mathcal{V}$ of associated vertex subsets	
	as follows:
	\begin{itemize}
		\item[(a)] Let $s\leq 3m$. Then take 
		$\vec{H} := \vec{H}_t(b)$.
		Replace all vertices $x$ of $\vec{H}$ with
		pairwise disjoint sets $V_x$ of size $m$ each,
		and replace every directed edge $(x,y)$
		with a hyperedge that contains
		$V_x$ and $V_y$ plus $s-2m$ further vertices
		that are not used for any other hyperedge.
		Call the resulting $s$-uniform hypergraph 
			$\mathcal{H}(m,b,s,t)$, and 
		set $\mathcal{V}=\{V_x:~x\in V(\vec{G})\}$
		to be its family of associated sets.
		
		\item[(b)] Let $s > 3m$, the hypergraph 
		$\mathcal{H}(m,b,s,t)$ is constructed as follows: 
		we take a set $V$ of $m$ vertices 
		and $b+1$ copies of $\mathcal{H}(m,b,s-m,t-1)$
        that are disjoint from each other and from $V$.	
        Then we add $V$ to every hyperedge of every copy of 
        $\mathcal{H}(m,b,s-m,t-1)$.
        Call the resulting $s$-uniform hypergraph 
			$\mathcal{H}(m,b,s,t)$.
		Moreover, let $\mathcal{V}$  
		contain $V$ and all associated sets from all 
		the $b+1$ copies of 
		$\mathcal{H}(m,b,s-m,t-1)$ used in this construction.
	\end{itemize}

\end{definition}

Figure~\ref{fig:hyper-constructions} shows how a copy of $\vec{G}_3(2)$ contained in $\vec{H}_5(2)$ is transformed
into a subhypergraph of $\mathcal{H}(5,2,14,5)$, according to the above definition.

\begin{figure}[ht]
    \centering
\includegraphics[scale=0.7,page=2]{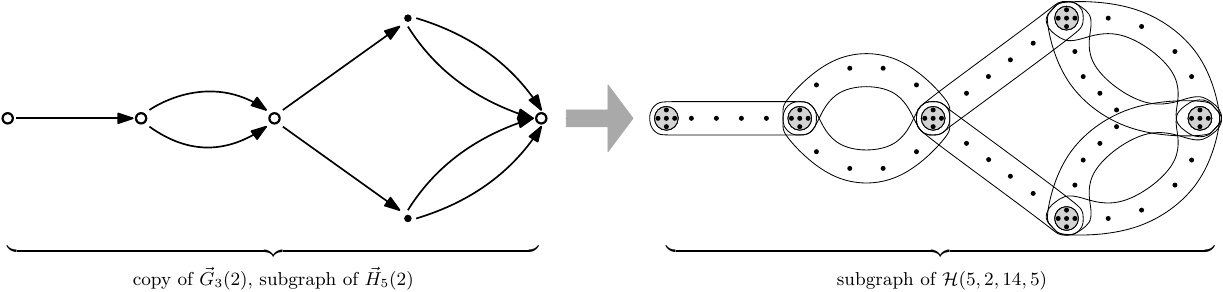}
    \caption{Replacing directed edges with hyperedges in Definition~\ref{def:hypergraph.H}(a)}
    \label{fig:hyper-constructions}
\end{figure}

The following can be seen with an easy induction on $s$.

\begin{observation}\label{obs:branched.path.overlap}
	Let $m,b,s,t \in \mathbb{N}$ with $s \geq 2m+1$, $m \leq b$, and $t \geq \lceil \frac{s}{m} \rceil$.
	The hypergraph $\mathcal{H}(m,b,s,t) = (X, \mathcal{F})$ 
	fulfils the following properties:
	\begin{enumerate}
		\item For any $V \in \mathcal{V}$ and any $F \in \mathcal{F}$ we have that $V \subset F$ or $V \cap F = \varnothing$.
		\item For each vertex $x \in X \setminus \bigcup_{V \in \mathcal{V}} V$ we have that $x$ is in at most one hyperedge from $\mathcal{F}$.
	\end{enumerate}
\end{observation}

Due to these properties we can allow ourselves to make some simple assumptions on Maker's strategy when playing on $(X,\mathcal{F})$, as summarized by the following proposition. 

\begin{proposition}\label{prop:branched.path.maker.plays.on.v}
	Let $m , b \in \mathbb{N}$ with $m\leq b$, and let a hypergraph $\mathcal{H} = (X,\mathcal{F})$ with a collection $\mathcal{V}$ of disjoint vertex-sets fulfilling the following properties be given:
	\begin{enumerate}
		\item $|F| > m$ for all $F \in \mathcal{F}$.
		\item $|V| = m$ for all $V \in \mathcal{V}$.
		\item For any $V \in \mathcal{V}$ and any $F \in \mathcal{F}$ we have that $V \subset F$ or $V \cap F = \varnothing$.
		\item For each vertex $x \in X \setminus \bigcup_{V \in \mathcal{V}} V$ we have that $x$ is in at most one hyperedge from $\mathcal{F}$.
	\end{enumerate}
	If Maker has a winning strategy playing an $(m:b)$-game on $\mathcal{H}$, she also has a winning strategy playing an $(m:b)$-game on $\mathcal{H}$ where in each of her moves she either claims all vertices of a set $V \in \mathcal{V}$, or fully claims all remaining vertices of a winning set $F \in \mathcal{F}$ to immediately win the game.
\end{proposition}

\begin{proof}
	Let $M_i$ denote the set of vertices Maker takes in the $i$-th round of the game when she plays her winning strategy, and let $\mathcal{F}_i$ denote all hyperedges which are not touched by Breaker before the $i$-th round. Let $j$ be the smallest integer such that $M_j \notin \mathcal{V}$. Assume Maker did not win after $j$ moves. We aim to show that 
Breaker can claim a set $B_j$ such that for all $F \in \mathcal{F}_j$ with $F \cap M_j \neq \varnothing$ we also have $F \cap B_j \neq \varnothing$. Note that $|B_j| \geq |M_j|$ since $b \geq m$. Thus, it suffices to show that for each $x \in M_j$
Breaker can claim a vertex $y \in X$  such that $y \in F$ for all $F \in \mathcal{F}_j$ with $x \in F$. If $x \notin \bigcup_{V \in \mathcal{V}} V$, then there is at most one $F \in \mathcal{F}_j$ with $x \in F$ by (4). Since Maker did not win with her $j$-th move, there are still vertices of $F$ available, and Breaker can claim an arbitrary vertex of those. Otherwise, there is $V \in \mathcal{V}$ with $x \in V$. Since $M_j \notin \mathcal{V}$, and $M_i \in \mathcal{V}$ for all $i < j$, Maker cannot have claimed all vertices of $V$ yet. If Breaker has not already claimed a vertex of $V$, he can take an arbitrary $y \in V \setminus M_j$. By (3), $y \in F$ for all $F \in \mathcal{F}_j$ with $x \in F$ as required.
	
	Hence, Maker is not able to complete any winning set that contains a vertex from $M_j$, and she also wins the game if she skips her $j$-th move. Repeating this argument gives the desired winning strategy.
\end{proof}

In particular, we can deduce the following main lemma of this section.

\begin{lemma}\label{lemma:speed.on.hypergraph}
	Let $m,b,s,t \in \mathbb{N}$ with $s \geq 2m+1$, $m \leq b$ and $t \geq \lceil \frac{s}{m} \rceil$, and let 
	$\mathcal{H} = (\mathcal{X},\mathcal{F}) := \mathcal{H}(m,b,s,t)$ be given by Definition~\ref{def:hypergraph.H}. Then the following holds for the $(m:b)$ Maker-Breaker game on $\mathcal{H}$:
	\begin{enumerate}
		\item Maker has a strategy as the starting player to win in $t$ rounds.
		\item Breaker has a winning strategy if he can claim a single vertex before Maker starts with her first move.
		\item Breaker has a strategy to prevent Maker from winning within $t-1$ rounds, even if Maker starts the game. 
	\end{enumerate}
\end{lemma}

\begin{proof}
We prove the statement by induction on $s$.

Assume first that $s\leq 3m$. Let $\mathcal{V}$ denote the family of associated vertex sets for
$\mathcal{H}(m,b,s,t)$, which is constructed from $\vec{H}_t(b)$
such that $\mathcal{V} = \{V_x:~ x\in V(\vec{H}_t(b))\}$. 
By Proposition~\ref{prop:branched.path.maker.plays.on.v}
we can assume that Maker always claims all vertices of a set $V\in \mathcal{V}$, or fully claims all remaining vertices of a winning set $F \in \mathcal{F}$ to win the game.
This can be represented as Maker claiming vertices of 
$\vec{H}_t(b)$, or claiming a directed edge $(v,w)$ to immediately win the game on $\vec{H}_t(b)$.
Moreover, whenever Breaker claims a vertex $x$ of $\mathcal{H}(m,b,s,t)$, this can be represented as Breaker claiming a vertex from $\vec{H}_t(b)$ (if $x\in V$ for some $V\in \mathcal{V}$)
or an edge of $\vec{H}_t(b)$ (if $x$ does not belong to a set $V\in \mathcal{V}$), and vice versa. Therefore, the statements (1)--(3) in Lemma~\ref{lemma:speed.on.hypergraph}
follow from the properties (1)--(3) in
Lemma~\ref{lem:game.directed.H}.

\smallskip

Assume then that $s> 3m$. 
To prove (1), Maker claims in her first move all vertices of the set $V$ described in part (b) of Definition~\ref{def:hypergraph.H}. After Breaker's move, there has to be a copy $\mathcal{H}'$ of $\mathcal{H}(m,b,s-m,t-1)$ which is not touched by Breaker yet, and thus, Maker wins on $\mathcal{H}'$ in $t-1$ further rounds by induction.
To prove (2), Breaker chooses for his first single element any vertex from the set $V$ described in part (b) of Definition~\ref{def:hypergraph.H}. Since every vertex of $V$ is part of every hyperedge, Maker cannot win anymore.

Finally, we consider (3). Maker has to claim all vertices of $V$ in her first move. Otherwise, Breaker would be able to claim a single vertex from $V$, and thus wins as already explained. Afterwards, Breaker can skip his first move and still achieve his goal by playing according to the strategy promised by (3) given by the induction hypothesis to prevent Maker from winning in less than $t - 2$ further rounds on any copy of $\mathcal{H}(m,b,s-m,t-1)$. Hence, in total Maker does not win within $t-1$ rounds.
\end{proof}

\begin{remark}\label{remark:skipping.move}
Note that the strategies given for property (3) in Lemma~\ref{lem:game.directed.H} and Lemma~\ref{lemma:speed.on.hypergraph} have the following property: depending on Maker's first move, Breaker can either skip his first move and still ensure that the game lasts at least $t-1$ rounds, or Breaker can prevent Maker from winning.
\end{remark}

Having the last lemma in our hands, we are now ready to prove our main theorems.

\subsection{Proof of Theorem~\ref{thm:construction.MB.time.size}}

Let $m,b,s,s',t,t'$ be given as stated in Theorem~\ref{thm:construction.MB.time.size}. 
We use the following construction: $\mathcal{H}$ consists disjointly of $b$ copies of $\mathcal{H}(m,b,s,t)$ denoted by $\mathcal{H}_i, i \in [b]$, and a single copy of $\mathcal{H}(m,b,s',t')$ denoted by $\mathcal{H}_{b+1}$.
 By Observation~\ref{obs:branched.path.overlap} we see that $\mathcal{H}$ fulfils the requirements of Proposition~\ref{prop:branched.path.maker.plays.on.v}, and hence we can assume that Maker in each of her moves claims only elements that belong to the same winning set when playing on $\mathcal{H}$.

First, if Maker starts the game, she has a strategy to claim a winning set of size $s$ in $t$ rounds on $\mathcal{H}_1$. She cannot claim a smaller winning set since every hyperedge has at least $s$ vertices. Additionally, Maker cannot win faster than within $t$ rounds on any single copy $\mathcal{H}_i, i \in [b+1]$ due to Property (3) of Lemma~\ref{lemma:speed.on.hypergraph}. Because of Proposition~\ref{prop:branched.path.maker.plays.on.v}, Maker claims in each of her move only elements from a single $\mathcal{H}_i, i \in [b+1]$, and she cannot win faster on $\mathcal{H}$ than on a single copy $\mathcal{H}_i, i \in [b+1]$.

Second, when Breaker starts the game, he can claim a single vertex from $\mathcal{H}_i$ for each $i \in [b]$, and thereby is able to prevent Maker from winning on those copies because Maker always needs to play in a single copy $\mathcal{H}_i, i \in [b]$ due to Proposition~\ref{prop:branched.path.maker.plays.on.v}, and Breaker can respond on that copy by the strategy guaranteed by Lemma~\ref{lemma:speed.on.hypergraph}(2). Afterwards, Maker can win on $\mathcal{H}_{b+1}$ within $t'$ rounds (but not faster) claiming a winning set of size $s'$ (but not smaller), again by Lemma~\ref{lemma:speed.on.hypergraph}. \qed

\subsection{Proof of Theorem~\ref{thm:construction.Gledel.improved}}

Let $m,b,r,s,s',t,t'$ be given as stated in Theorem~\ref{thm:construction.Gledel.improved}. 
Let
$\mathcal{H}$ be the hypergraph constructed in the previous proof, and obtain $\mathcal{H}'$ by adding disjointly a single hyperedge $F_\gamma$ with $|F_\gamma| = r$. 
Since $r>m$, the additional hyperedge does not help Maker
claim a winning set, and hence $\mathcal{H}'$ also satisfies the properties (1)
and (2) from Theorem~\ref{thm:construction.MB.time.size}.
Next, let $G$ be obtained from Lemma~\ref{lemma:from.H.to.G}
with input $\mathcal{H}'$. 
First, $\gamma(G) = r$, due to Claim~\ref{clm:transference.claim2}
and since the smallest winning set of $\mathcal{H}'$ is $F_\gamma$.
Second, due to Lemma~\ref{lemma:from.H.to.G} and the properties of $\mathcal{H}'$, we obtain
 $s_{MB}(G,m:b) = s$,
 $\gamma_{MB}(G,m:b) = t$, 
 $s_{MB}'(G,m:b) = s'$ and $\gamma_{MB}'(G,m:b) = t'$. \qed

\subsection{Proof of Theorem~\ref{thm:construction.large.before.small}}
Let $m,b,s,t$ be given as stated in Theorem~\ref{thm:construction.large.before.small}. We first construct a hypergraph $\mathcal{H}$ in the following way: We take a copy of $\mathcal{H}(m,b,s,\lceil \frac{t}{m} \rceil + 1)$ on a vertex set $X$, and add a hyperedge for every subset of $X$ of size $t$. Then every winning set in the game on $\mathcal{H}$
is either of size $s$ or of size $t$. According to Lemma~\ref{lemma:speed.on.hypergraph} Maker can claim a winning set of size $s$ in $\lceil \frac{t}{m} \rceil + 1$ rounds, but not faster.
Moreover, after $\lceil \frac{t}{m} \rceil$ rounds, Maker occupies a winning set of size $t$, independent of how she plays, because every set of $t$ vertices forms a winning set.

Next, let $G$ be the result of calling Lemma~\ref{lemma:from.H.to.G} with input $\mathcal{H}$. We aim to show that $G$ satisfies the required properties. By Lemma~\ref{lemma:from.H.to.G}
we immediately have
$s_{MB}(G,m:b) = s$. 
Hence, it remains to show that
Dominator cannot claim a dominating set of size $s$ before having a
minimal dominating set of size $t$.
To do so, fix any Dominator's strategy.
Then let Staller play on $G$ according to the strategy given in the proof of 
Lemma~\ref{lemma:from.H.to.G},
where $\mathcal{S}_B$ is a Breaker's strategy
that prevents Maker's win
on $\mathcal{H}(m,b,s,\lceil \frac{t}{m} \rceil + 1)$ 
for $\lceil \frac{t}{m} \rceil$ rounds.
Assume that Dominator at some point occupies a dominating set $D$ of size $s$. As explained 
in the proof of 
Lemma~\ref{lemma:from.H.to.G},
the set $D'=D\cap X$ must then be
a dominating set of $G$, and hence must contain a winning set $F$ of $\mathcal{H}$, because of 
Claim~\ref{clm:transference.claim3}.
However, as before, by using the strategy $\mathcal{S}_B$ and since $|F| \leq |D'|\leq s$, Dominator needs to play at least $\lceil \frac{t}{m} \rceil + 1$ rounds on $X$ 
in order to occupy such a set $D'$, and thus claim at least $t + m$ elements of $X$.
The set of the first $t$ vertices 
Dominator claims on $X$ then forms a winning set of $\mathcal{H}$
and hence a dominating set of $G$ by Claim~\ref{clm:transference.claim1}.
This dominating set must be minimal,
as it would otherwise need to contain a dominating set of size $s$; but Dominator cannot occupy such a set within $\lceil \frac{t}{m} \rceil$ rounds on $X$.
\qed

\subsection{Proof of Theorem~\ref{thm:construction.large.faster.than.small}}
Let $m,b,s,s',t,t'$ be given as stated in Theorem~\ref{thm:construction.large.faster.than.small}. We construct a hypergraph $\mathcal{H}$ as the disjoint union of a copy $\mathcal{H}_1$ of $\mathcal{H}(m,b,s',t)$ and a copy $\mathcal{H}_2$ of $\mathcal{H}(m,b,s,t')$. Let $G$ be the result of calling Lemma~\ref{lemma:from.H.to.G} with input $\mathcal{H}$. We aim to show that $G$ fulfills the required properties. First note that $\mathcal{H}$ fulfills the requirements of Proposition~\ref{prop:branched.path.maker.plays.on.v}. Thus, Maker can claim a winning set of size $s'$ in $t$ rounds but not faster by playing according to Lemma~\ref{lemma:speed.on.hypergraph} on $\mathcal{H}_1$. In the same way, she can claim a winning set of size $s$ in $t'$ rounds but not faster by playing on $\mathcal{H}_2$. Thus, we have $\gamma_{MB}(G,m:b)=t$ by Lemma~\ref{lemma:from.H.to.G} and $s_{MB}(G,m:b)=s$ by the same lemma and the fact, that $\mathcal{H}$ does not contain any winning sets of size smaller than $s$. 

Now, consider the following Breaker strategy on $\mathcal{H}$: Whenever Maker plays on some $\mathcal{H}_i$, Breaker will answer on that same $\mathcal{H}_i$ with the strategy given in either Lemma~\ref{lemma:speed.on.hypergraph}(2) or (3), depending on Maker's first move. Let $i \in [2]$ and $j \in [2] \setminus \{i\}$. Assume Maker's first move is played on $\mathcal{H}_i$. Breaker responds on $\mathcal{H}_i$ according to the strategy given in the proof of Lemma~\ref{lemma:speed.on.hypergraph}(3) which 
by Remark~\ref{remark:skipping.move} means that he will either prevent Maker from winning at all on $\mathcal{H}_i$, or he can skip his first move, depending on Maker's first move. If he can skip his first move, he instead claims a single element of $\mathcal{H}_j$ according to the strategy from Lemma~\ref{lemma:speed.on.hypergraph}(2) and keeps responding according to that strategy whenever Maker plays on $\mathcal{H}_j$. Otherwise if he cannot skip his first move on $\mathcal{H}_i$, he will respond on both $\mathcal{H}_i$ and $\mathcal{H}_j$ according to Lemma~\ref{lemma:speed.on.hypergraph}(3) whenever Maker plays on the respective hypergraph. Note that by applying this strategy, Breaker prevents Maker from winning on either $\mathcal{H}_1$ or $\mathcal{H}_2$, again depending on Maker's first move.

By following this strategy, Breaker ensures that Maker does not claim a winning set of size smaller than $s$ because such a set does not exist, and that she cannot win within $t - 1$ rounds because Breaker either wins on $\mathcal{H}_1$ or prevents a Maker win for at least $t - 1$ rounds by Lemma~\ref{lemma:speed.on.hypergraph}(3). In the same way, Maker cannot win on $\mathcal{H}_2$ within $t' - 1$ rounds. Additionally, Maker can only win on one of the copies, $\mathcal{H}_1$ or $\mathcal{H}_2$. If she wins within less than $t'$ rounds, she has to do so on $\mathcal{H}_1$, but then she cannot win on $\mathcal{H}_2$. Thus, the smallest winning set she can achieve is of size $s'$ in that case since $\mathcal{H}_1$ does not contain any winning sets of smaller size.

Regarding the domination game on $G$, Staller can play according to the strategy given in the proof of Lemma~\ref{lemma:from.H.to.G} with $\mathcal{S}_B$ being the described strategy on $\mathcal{H}$. By doing so, Staller can ensure the same properties as Breaker in the game on $\mathcal{H}$ following a similar argument as in the proof of Theorem~\ref{thm:construction.large.before.small}.
\qed

\medskip

\subsection{Proof of Theorem~\ref{thm:construction.MB.not.monotone}}

Let $B=\{b_1,b_2,\ldots,b_{\ell}\}$ and assume
$b_1<b_2<\ldots<b_{\ell}$.
For each $i\in [\ell +1]$, we first take a disjoint set $V_i$ of vertices such that the following size condition holds (where we set $b_0:=-1$):
$$
|V_i| = b_j ~ \Leftrightarrow ~ b_{j-1} + 1 < i \leq b_j + 1 .
$$
Then define $\calH = (X,\calF)$ as the hypergraph which
satisfies
$$
X = \bigcup_{i\in [\ell +1]} V_i
~~ \text{and} ~~
\calF = \big\{ F\subset X:~ |F\cap V_i|=1~ \text{for every }i\in [\ell + 1] \big\} .
$$

We first prove that Breaker wins the $(b:b)$ game on $\calH$ 
if $b\in B$: Let $b=b_j$ for some $j\in [\ell]$.
We note that $|V_i|\leq b$ for every $i\leq b+1$.
After Maker's first move, at least one of the sets $V_i$ with
$i\leq b+1$ does not contain a vertex claimed by Maker,
and thus Breaker can fully claim such a set $V_i$ in her first move.
As every hyperedge of $\calH$ intersect $V_i$,
Breaker immediately wins the game. 

\smallskip

Next we prove that Maker wins the $(b:b)$ game on $\calH$ 
if $b\notin B$:
In her first move, Maker claims exactly one vertex of each of the sets
$V_i$ with $i\leq b$.
Note that $|V_i|>b$ for every $i> b$.
Hence, from now on, whenever Breaker claims an element
in one of these sets $V_i$ of which Maker still does not
have a vertex, Maker can immediately
afterwards claim an element of the same set $V_i$.
By following this strategy, Maker manages to claim a vertex in
each of the sets $V_i$. These vertices form a hyperedge of $\calH$,
and hence Maker wins. \hfill $\Box$


\section{Results for Waiter-Client domination games}\label{sec:WC.domination}

\subsection{Dominating trees in Waiter-Client games}

Before we prove Theorem~\ref{thm:WC.trees},
we first consider a suitable residue graph $R(G)$ 
for any graph $G$,
and with
Corollary~\ref{cor:residue2} we provide
a Waiter-Client version of Theorem~4.4 in~\cite{gledel2019maker}.
We start with the following lemma.

\begin{lemma}\label{lem:residue1}
Let $G$ be a graph, and let $vw\in E(G)$ such that
$d_G(v)=1$ and $d_G(w)=2$. Then
$$
\gamma_{WC}(G) = 1 + \gamma_{WC}(G-\{v,w\})
~~\text{and}~~
s_{WC}(G) = 1 + s_{WC}(G-\{v,w\}).
$$
\end{lemma}

\begin{proof}
We start with the discussion on $\gamma_{WC}(G)$.
Let $\mathcal{S}_W$ be a strategy for Dominator
which ensures her win in the Waiter-Client domination game 
on $G-\{v,w\}$ within $\gamma_{WC}(G-\{v,w\})$ rounds.
Moreover, let $\mathcal{S}_C$ be a strategy for Staller
which ensures that Dominator does not win the Waiter-Client domination game on $G-\{v,w\}$ within less than
$\gamma_{WC}(G-\{v,w\})$ rounds.

In order to prove     
$\gamma_{WC}(G) \leq 1 + \gamma_{WC}(G-\{v,w\})$,
consider the following strategy for Dominator.
In the first round, Dominator offers $v$ and $w$,
and no matter what Staller does, both vertices are dominated by Dominator. Afterwards, Dominator follows the strategy $\mathcal{S}_W$
in order to dominate the remaining vertices.
In total, this takes no more than $1 + \gamma_{WC}(G-\{v,w\})$ rounds.

In order to prove 
$\gamma_{WC}(G) \geq 1 + \gamma_{WC}(G-\{v,w\})$,
consider the following strategy for Staller.
As long as Dominator does not offer a vertex from $\{v,w\}$,
Staller plays according $\mathcal{S}_C$. 
If Dominator offers $\{v,w\}$ in the same round,
then Staller claim $w$ and gives $v$ to Dominator;
afterwards Staller continues according to $\mathcal{S}_C$.
This way, $v,w$ do not help Dominator to dominate a vertex in $G-\{v,w\}$, and hence, Dominator needs at least 
$1 + \gamma_{WC}(G-\{v,w\})$ rounds for winning.
Otherwise, i.e.~if $v$ and $w$ are offered in different rounds,
then Staller claims both $v$ and $w$ for herself.
In this case, Dominator loses as she cannot dominate $v$.

\smallskip

Next, we consider $s_{WC}(G)$.
In the above argument, replace $\mathcal{S}_W$ with a strategy for Dominator
which ensures that in the Waiter-Client domination game 
on $G-\{v,w\}$ she can get a dominating set of size $s_{WC}(G-\{v,w\})$,
and replace $\mathcal{S}_C$ with a strategy for Staller
which ensures that in the Waiter-Client domination game on $G-\{v,w\}$ Dominator does not get a dominating set of size $s_{WC}(G-\{v,w\}) - 1$.
Then, $s_{WC}(G) = 1 + s_{WC}(G-\{v,w\})$
can be deduced analogously.
\end{proof}

\smallskip

Now, as in~\cite{gledel2019maker},
we obtain $R(G)$ by iteratively deleting as long as possible
adjacent vertices $v,w$ such that
$v$ is a leaf in the current graph
and $w$ is a neighbour of $v$ with degree $2$.
Note that $R(G)$ is well-defined by Lemma~4.1
in~\cite{gledel2019maker}. Moreover,
the following is an easy consequence of Lemma~\ref{lem:residue1}.

\begin{corollary}\label{cor:residue2}
For every graph $G$, we have 
$$\gamma_{WC}(G) = \frac{v(G)-v(R(G))}{2} + \gamma_{WC}(R(G))
~~ \text{and} ~~
s_{WC}(G) = \frac{v(G)-v(R(G))}{2} + s_{WC}(R(G))
.$$
\end{corollary}
    
Finally, we can use this to prove Theorem~\ref{thm:WC.trees}.

\begin{proof}[Proof of Theorem~\ref{thm:WC.trees}]
Let $T$ be a tree and let $R(T)$ be its residue.
By Lemma~4.2 in~\cite{gledel2019maker}, $T$ has a perfect matching if and only if
$R(T)$ has a perfect matching. Moreover, by Lemma 21
in~\cite{duchene2020maker},
$R(T)$ is either (a) an isolated vertex, (b) an edge or (c) contains a vertex $x$ with at least two leaves as neighbours.

In case (a), $T$ has an odd number of vertices and hence no perfect matching. Moreover, $\gamma_{WC}(R(T))=\infty$ and $s_{WC}(R(T))=\infty$,
as by the definition of Waiter-Client games,
Staller obtains the unique vertex of $R(T)$ when playing on $R(T)$ only. Thus, $\gamma_{WC}(T)=\infty$ 
and $s_{WC}(T)=\infty$ by Corollary~\ref{cor:residue2}.

In case (b), $R(T)$ has a perfect matching
and hence $T$ has a perfect matching. Moreover, as 
$\gamma_{WC}(R(T))=s_{WC}(R(T))=1$, Corollary~\ref{cor:residue2} implies $\gamma_{WC}(T) = s_{WC}(T) = \frac{v(T)}{2}$.

In case (c), $R(T)$ does not have a perfect matching
and hence $T$ does not either. Staller wins the Waiter-Client domination game on $R(T)$, as a simple case distinction shows that he can claim
the vertex $x$ and at least one of its leaf neighbours.
This leaf neighbour then cannot be dominated by Dominator.
In particular, $\gamma_{WC}(R(T))=\infty$ and $s_{WC}(R(T))=\infty$, and hence
$s_{WC}(T)=\infty$ and $\gamma_{WC}(T)=\infty$ by Corollary~\ref{cor:residue2}.
\end{proof}

\subsection{Dominating cycles in Waiter-Client games}

\begin{proof}[Proof of Theorem~\ref{thm:WC.cycle}]
Let the vertices of $C_n$ be labeled $v_1,v_2,\ldots,v_n$
such that $v_iv_j$ forms an edge if $i-j \equiv \pm 1$ (mod $n$).
We first show $\gamma_{WC}(C_n)\leq \lfloor \frac{n}{2} \rfloor$ and $s_{WC}(C_n)\leq \lfloor \frac{n}{2} \rfloor$, by describing a Dominator strategy. In the first round, let Dominator offer $\{v_{n-1},v_n\}$. By symmetry we can assume that Dominator receives the vertex $v_{n-1}$.
Afterwards let $k$ be the largest even integer with $k\leq n-2$.
Playing on the path $(v_1,\ldots,v_k)$,
Dominator can dominate all its vertices within $\frac{k}{2}$ rounds
by Theorem~\ref{thm:WC.trees}. As $v_{n-2},v_{n-1},v_n$
are dominated by the vertex $v_{n-1}$ from the first round,
Dominator has a dominating set then. In total, she used
$\frac{k}{2}+1 = \lfloor \frac{n}{2} \rfloor$ rounds.

\smallskip

Next we show $\gamma_{WC}(C_n)\geq \lfloor \frac{n}{2} \rfloor$.
This inequality is easy for $n\in \{3,4,5\}$,
as in these cases the domination number of $C_n$ equals
$\lfloor \frac{n}{2} \rfloor$.
Hence, from now on we can assume that $n\geq 6$.
We consider two cases depending on the first move of Dominator.

Assume first that in the first round, Dominator offers two vertices which are not adjacent; by symmetry we may assume that she offers $v_1$ and
$v_i$ with $3\leq i\leq \lfloor \frac{n}{2} \rfloor +1 $. Then Staller can keep $v_1$ and give $v_i$ to Dominator.
Afterwards, Staller can answer almost arbitrarily to Dominator's offers, as long as he ensures the following:
\begin{itemize}
    \item[(a)] If Dominator offers $v_n$ and $v_2$ in different rounds,
      Staller keeps both of these vertices.    
    \item[(b)] If Dominator offers $v_n$ and $v_2$ in the same round,
      Staller keeps $v_n$ and gives $v_2$ to Dominator.
    \item[(c)] If Dominator offers $v_{n-1}$ together with a vertex $w\notin \{v_2,v_n\}$, then Staller keeps $v_{n-1}$.
\end{itemize}
It is easy to check that Staller can ensure all three conditions (a)--(c) to hold. If Dominator offers $v_2,v_n$ in different rounds, then by (a), the vertex $v_1$ cannot be dominated by Dominator.
Otherwise, by (b) and (c), the vertex $v_n$ cannot be dominated by Dominator. In any case, Dominator loses.

Assume then that in the first round, 
Dominator offers two adjacent vertices, w.l.o.g the vertices $v_{1}$ and $v_2$. Then Staller keeps $v_2$ and gives $v_{1}$ to Dominator. Afterwards, Staller answers to Dominator's offers as follows:
\begin{itemize}
    \item[(a)] If Dominator offers $\{v_{2i-1},v_{2i}\}$ for some
    $1\leq i\leq \lfloor \frac{n}{2} \rfloor - 1$, then Staller keeps $v_{2i}$ and gives $v_{2i-1}$ to Dominator.
    \item[(b)] Otherwise, if Dominator offers a pair of vertices different from those in (a), then Staller keeps the vertex of smaller index and gives the other to Dominator. 
\end{itemize}

Note that the first round was already played according to (a).
If throughout the game, Dominator always offers pairs as decribed in (a), she gets a dominating set, but then she plays $\lfloor \frac{n}{2} \rfloor$ rounds until she wins, and there is no dominating set with less
than $\lfloor \frac{n}{2} \rfloor$ vertices contained in Dominator's vertex set.
Otherwise, Dominator loses by the following reason:
Let $j\leq \lfloor \frac{n}{2} \rfloor - 1$ be the smallest index such that Dominator does not offer 
$\{v_{2j-1},v_{2j}\}$ in the same round. Then by following (b),
Staller claims both $v_{2j-1},v_{2j}$ for himself.
Moreover, as $v_{2j-3},v_{2j-2}$ 
were offered in the same round (by minimality of $j$), 
Staller also keeps $v_{2j-2}$ (because of (a) or the first round),
and hence, Dominator cannot dominate $v_{2j-1}$.
\end{proof}

\newpage

\subsection{Comparing Maker-Breaker and Waiter-Client domination games}

\begin{proof}[Proof of Theorem~\ref{thm:construction.MB.and.WC}]
Due to Lemma~\ref{lemma:from.H.to.G}, it is enough to find
a hypergraph $\calH_{s,t}$ such that
\begin{itemize}
\item[(i)] $s$ is the smallest number of rounds in which Maker wins the unbiased Maker-Breaker game on $\calH_{s,t}$ when Breaker starts, and
\item[(ii)] $t$ is the smallest number of rounds
in which Waiter wins the unbiased Waiter-Client game on $\calH_{s,t}$.
\end{itemize}

\textbf{Case 1:} We start with the case $s\geq t$ (for which it is actually enough to require that $t\geq 3$) and prepare for the desired construction 
with the following claim.

\begin{claim}\label{claim.WC.yes.MB.no}
For every $t\geq 3$ there is a hypergraph $\calH_{t}$
such that Maker looses on $\calH_t$,
and such that $t$ is the smallest number of rounds within which 
Waiter wins on $\calH_t$.
\end{claim}

\begin{proof}
Let $a_i,b_i\notin [2t-6]$ with $i\in [4]$ be distinct elements,
and define $\calH_t = (X_t,\calF_t)$ such that
$\calF_t := \{ \{a_i,b_i\}\cup M: i\in [4],~ M\in \binom{[2t-6]}{t-3} \}$.
By playing a pairing strategy for the pairs $\{a_i,b_i\}$,
Breaker can prevent Maker from winning.
Waiter can win within $t$ rounds as follows:
In the first two rounds, Waiter offers $a_1,a_2,a_3,a_4$
and w.l.o.g.~Waiter gets $a_1,a_2$.
Then she offers $b_1,b_2$ and w.l.o.g receives $b_1$.
For the next $t-3$ rounds, Waiter offers the elements of
$[2t-6]$ in an arbitrary way, and completes a winning set
of the form $\{a_1,b_1\}\cup M$ with 
$M\in \binom{[2t-6]}{t-3}$.
Moreover, Client can prevent Waiter from winning in less than $t$ rounds as follows: In each round, in which Waiter offers at least one element in $[2t-6]$, Client gives an element of $[2t-6]$ to Waiter. Note that Waiter needs to play at least $t-3$ such rounds.
In order to complete a winning set, Waiter moreover needs
a pair $\{a_i,b_i\}$. W.l.o.g. let $a_1$ be the first element from
$A:=\{a_i,b_i:~i\in [4]\}$ that Waiter occupies (by offering a pair in $A$), then when Waiter offers the next pair in $A$,
Client only needs to make sure that Waiter does not get $b_1$.
This way, Waiter needs to offer at least 3 pairs in $A$
until she wins, and hence needs to play at least $t$ rounds in total.
\end{proof}

Next, we set $\calH_{s,t} := \calH_t \dot\cup K_{2s}^{(s)}$
where $K_{2s}^{(s)}$ is the $s$-uniform complete hypergraph on
$2s$ vertices. By Claim~\ref{claim.WC.yes.MB.no},
the elements of $\calH_t$ do not help Maker in winning,
but Maker can obviously win within $s$ rounds (but not faster)
by playing on $K_{2s}^{(s)}$. This shows (i).
By Claim~\ref{claim.WC.yes.MB.no}, Waiter can win on $\calH_{s,t}$ within $t$ rounds by playing on $\calH_t$ only (and here she cannot be faster). Moreover, the vertices of $K_{2s}^{(s)}$
do not help her to be faster, as every winning set in 
$K_{2s}^{(s)}$ has size $s\geq t$. This proves (ii).

\medskip

\textbf{Case 2:} We then consider the case $s < t$
and prepare for the desired construction 
with the following claim. For the proof, note that
a Client-Waiter game on a hypergraph $\calH$ is played
the same way as a Waiter-Client game with the difference that
Client wins if and only if he obtains a winning set.
(Again, if there is only one element in the last round,
then it is given to Client.)

\begin{claim}\label{claim.MB.yes.WC.no}
For every $s\geq 7$ there is a hypergraph $\calH'_{s}$
such that Waiter looses on $\calH'_s$,
and such that $s$ is the smallest number of rounds within which 
Maker wins on $\calH'_s$ (if Breaker starts).
\end{claim}

\begin{proof}
In \cite{knox2012two} (see Construction 5),
there is a hypergraph $G_{CP} = (X_{CP},\calF_{CP})$ on 15 vertices given
with the following properties:
\begin{itemize}
\item[(i')] Breaker wins the Maker-Breaker game on $G_{CP}$,
if Maker is the first player (Proposition 6 in~\cite{knox2012two}), and this takes at most $7$ rounds.
\item[(ii')] Client wins the Client-Waiter game on $G_{CP}$
(Proposition 7 in~\cite{knox2012two}).
\end{itemize}
Now, consider the transversal hypergraph $\calH_0 = (X_{CP},\calF_0)$, where
$$\calF_0:=\{F\subset X_{CP}:~ (\forall F'\in \calF_{CP}:~ F\cap F'\neq \varnothing)\}.$$
Then the above properties can be reformulated as follows:
\begin{itemize}
\item[(i'')] Maker wins the Maker-Breaker game on $\calH_0$,
if Breaker is the first player, and this takes at most $7$ rounds.
\item[(ii'')] Client wins the Waiter-Client game on $\calH_0$.
\end{itemize}
Let $r\leq 7$ be the smallest number of rounds within which
Maker can always win on $\calH_0$.
W.l.o.g assume that $X_{CP}\cap \mathbb{N}=\varnothing$.
We now consider the hypergraph $\calH'_s =(X_s,\calF_s)$ 
with $X_s := X_{CP} \cup [2s - 2r + 7]$, and
$$
\calF_s :=
\left\{ F\cup M:~ F\in \calF_0,~ M\in \binom{[2s-2r+7]}{s-r} \right\}. 
$$
Then Client wins on $\calH_s'$ as follows:
Whenever Waiter offers two elements in $X_{CP}$,
Client plays according to the strategy promised by (ii''),
and in each other case she never gives an element of $X_{CP}$
to Waiter. This way, Waiter never occupies a set from
$\calF_0$ and in particular no set in $\calF_s$.
Maker wins on $\calH_s'$ within $s$ rounds as follows:
She first plays $r$ rounds on $X_{CP}$ to occupy a
set $F\in \calF_0$, which is possible by (i''),
and then she plays $s-r$ further rounds in which she claims
arbitrary elements of $[2s-2r+7]$, which is possible as before that at most $7$ rounds were played.
Moreover, Breaker can prevent Maker from winning faster by
the following reason: Maker needs to claim a set
$F\cup M$ with $F\in \calF_0$ an $M\in \binom{[2s-2r+7]}{s-r}$.
By definition of $r$, 
Breaker can play so that for a suitable set $F$,
Maker needs to play at least $r$ rounds in $X_{CP}$,
and additionally Maker needs $s-r$ elements outside $X_{CP}$ for a suitable set $M$. 
\end{proof}

Using the above claim, we set 
$\calH_{s,t} := \calH_s' \dot\cup K_{2t}^{(t)}$.
Now, showing that (i) and (ii) hold can be done analogously 
to Case 1.
This completes the proof of Theorem~\ref{thm:construction.MB.and.WC}. 
\end{proof}


\section{Concluding remarks and open problems}\label{sec:concluding}

\textbf{Optimality.}
Let us first look at the conditions assumed in
Theorem~\ref{thm:construction.MB.time.size} and
Theorem~\ref{thm:construction.Gledel.improved}.
Since Maker as first player can always achieve
what Maker as second player can do, the inequalities
$s'\geq s$ and $t'\geq t$ need to be assumed in both theorems.
Similarly, since claiming a set of size $x$ with bias $m$
takes at least $\lceil \frac{x}{m} \rceil$ rounds,
the conditions $t\geq \lceil \frac{s}{m} \rceil$ and
$t'\geq \lceil \frac{s'}{m} \rceil$ need to be assumed as well.
Additionally, assuming $s\geq r$ in Theorem~\ref{thm:construction.Gledel.improved} comes from the fact that Maker cannot claim a winning set which is smaller than the smallest winning set, and $m\leq r-1$ is assumed since otherwise Maker would be able to win the game in the first round.
The only conditions which may not be optimal in both theorems, 
but which are required for our constructions, 
are $s\geq 2m+1$ and $m\leq b$. However, we note that these conditions are still optimal for $m=1$, which can be seen as follows. The inequality $1=m\leq b$ is clearly satisfied then.
If we would have $s\leq 2m$ in Theorem~\ref{thm:construction.MB.time.size}, then $s\in \{1,2\}$.
For property (1) in Theorem~\ref{thm:construction.MB.time.size} being true,
the hypergraph $\mathcal{H}$ would need to contain at least one winning set of size at most $s\leq 2$. If there is a winning set of size $1$, Maker can win in the first round and hence, property (1) cannot be true for arbitrary $t$. Otherwise, let $G$ be the graph induced by all hyperedges of $\mathcal{H}$ of size $2$. If 
the maximum degree $\Delta(G)$ of $G$ satisfies $\Delta(G)\geq b+1$, then Maker can occupy an edge of $G$ within 2 rounds, and therefore, $t$ cannot be chosen arbitrarily in property (1). 
If instead $\Delta(G)\leq b$, then in each of his moves Breaker can block all edges which contain the vertex previously claimed by Maker,
in contradiction to Maker being able to occupy a winning set of size $2$. The discussion for Theorem~\ref{thm:construction.Gledel.improved} is analogous.
Moreover, a similar discussion shows that, for general $m\in\mathbb{N}$, the assumption $s\geq 2m+1$ cannot be replaced with $s\geq m$. To close this gap is an open problem.

\begin{problem}\label{problem:smaller.sets}
Can the condition $s\geq 2m+1$ in Theorem~\ref{thm:construction.MB.time.size} and
Theorem~\ref{thm:construction.Gledel.improved} be replaced
with $s\geq 2m$? If yes, what is the best bound on $s$
that can be assumed here?
\end{problem}

Moreover, the above discussion can be modified to show that $s \geq 2m$ is actually not possible if the construction should fulfill the requirements of Proposition~\ref{prop:branched.path.maker.plays.on.v}. Thus, Maker's strategy to win on a construction that positively answers Problem~\ref{problem:smaller.sets} would necessarily include her claiming elements from different winning sets in some of her moves. Similarly, Proposition~\ref{prop:branched.path.maker.plays.on.v} does not work if $m > b$. In fact, if $m > b$, Maker can win on each $s$-uniform hypergraph $\mathcal{H}$ that contains a large enough matching in a number of moves depending on $s$ using a box game argument. Therefore, if the number of required moves should be arbitrarily large compared to $s$, $\mathcal{H}$ cannot contain a large matching in that case.
\medskip

\textbf{Waiter-Client domination games.}
In this paper we give first results on Waiter-Client domination games. It seems natural to study analogues of the questions discussed for Maker-Breaker domination games in e.g.~\cite{bagdas2024biased,brevsar2025thresholds,
bujtas2025criticality,divakaran2025maker,dokyeesun2023maker,duchene2020maker}. 
For instance: 
\begin{itemize}
\item Is deciding the winner of the Waiter-Client domination game  PSPACE-complete? 
\item What can be said about biased Waiter-Client domination games? 
\item What happens if Waiter wants to occupy a total dominating set? \end{itemize}

\medskip

\textbf{Random graphs.} As already observed in~\cite{bagdas2024biased}, 
it is not hard to show that with high probability,
if $G$ is a Binomial random graphs $G_{n,p}$ with
constant edge probability $p\in (0,1)$ 
or a random geometric graphs $\mathcal{G}_2(n,r)$ with a radius 
$r\in [0,1]$, 
the parameters $\gamma_{MB}(G)$ and $\gamma'_{MB}(G)$ 
asymptotically equal the domination number of $G$. 
We conjecture that in these cases the two game parameters 
$\gamma_{MB}$ and $s_{MB}$ are equal.

\begin{conjecture}
 Let $G$ be a Binomial random graph $G_{n,p}$ with constant edge probability $p\in (0,1)$,
 then $\gamma_{MB}(G,b) = s_{MB}(G,b)$
 holds with high probability.
\end{conjecture}

\begin{conjecture}
 Let $G$ be a random geometric graph $\mathcal{G}_2(n,r)$ with constant radius $r\in (0,1)$,
 then $\gamma_{MB}(G,b) = s_{MB}(G,b)$
 holds with high probability.
\end{conjecture}

\section*{Acknowledgement}
This topic was started
in connection with a Bachelor thesis of the first author at the Hamburg University of Technology. He wants to thank Anusch Taraz and his research group for the supervision.
Moreover, this research was started when the second and fourth author were supported 
by Deutsche Forschungsgemeinschaft (Project CL 903/1-1). Additionally, the research of the fourth author was partly supported by the ANR project P-GASE (ANR-21-CE48-0001-01).

\bibliographystyle{amsplain}
\bibliography{references}

\bigskip\bigskip\bigskip

\address{Hamburg University of Technology, Germany.} \\
{\scriptsize \email{ali.bagdas@tuhh.de, dennis.clemens@tuhh.de, fabian.hamann@tuhh.de}} \\[2pt]
\address{Universit\'e Claude Bernard Lyon 1, France.}\\
{\scriptsize \email{yannick.mogge@univ-lyon1.fr}}

\end{document}